\documentclass[11pt]{article}
\usepackage{fullpage}
\usepackage{amsmath}
\usepackage{graphicx}
\usepackage{subcaption}
\usepackage{bbm}
\usepackage{stmaryrd}
\usepackage{hyperref}
\usepackage{array}
\RequirePackage{amssymb}
\usepackage{amsthm}
\usepackage{caption}
\usepackage{color}

\numberwithin{equation}{section}
\allowdisplaybreaks[1]

\newtheorem{theorem}{Theorem}
\newtheorem{lemma}{Lemma}
\theoremstyle{remark}
\newtheorem{remark}{Remark}
\newcommand{\OU}{Ornstein-Uhlenbeck }
\newcommand{\R}{\mathbb{R}}
\newcommand{\overbar}[1]{\mkern 4mu\overline{\mkern-4mu#1\mkern-4mu}\mkern 4mu}
\newcommand{\Rbar}{\overbar{\R}}
\newcommand{\1}{\mathbbm{1}}
\newcommand{\vn}[1]{\left|\left|#1\right|\right|}
\newcommand{\bo}[1]{O\left(#1\right)}

\newcommand{\bth}[1]{\Theta\left(#1\right)}
\newcommand{\PP}{\mathbb{P}}
\newcommand{\de}{\,{\buildrel d \over =}\,}

\title{Join the Shortest Queue with Many Servers. \\ The Heavy Traffic Asymptotics}
\author{Patrick Eschenfeldt \\ MIT Operations Research Center \and David Gamarnik \\ MIT Sloan School of Management and Operations Research Center}
\date{}

\begin{document}

\maketitle

\abstract
{
 We consider queueing systems with $n$ parallel queues under a Join the Shortest Queue (JSQ) policy in the Halfin-Whitt heavy traffic regime. We use the martingale method to prove that a scaled process counting the number of idle servers and queues of length exactly 2 weakly converges to a two-dimensional reflected Ornstein-Uhlenbeck process, while processes counting longer queues converge to a deterministic system decaying to zero in constant time. This limiting system is comparable to that of the traditional Halfin-Whitt model, but there are key differences in the queueing behavior of the JSQ model. In particular, only a vanishing fraction of customers will have to wait, but those who do will incur a constant order waiting time. }

\section{Introduction.}
In this paper we consider queueing systems with many parallel servers under a heavy-traffic regime where the workload scales with the number of servers. Such systems are well understood when a global queue is maintained (i.e.\ $M/M/n$ and similar queues \cite{halfin-whitt}), but in many practical situations it may be advantageous to instead maintain parallel queues. Even if a global queue is itself not problematic, it may be necessary to keep queued customers close to the server who will eventually serve them. Consider for example an airport setting with arriving passengers who need to have their passports checked with one of a large number of passport controllers. In this situation having only a global queue can lead to significant walk times between the front of the queue and the server, leaving servers idle while they wait for their next customer. This idle time can be avoided by routing customers to individual queues for each server \emph{before} earlier customers finish service.

At the same time, a parallel scheme will necessarily allow servers to idle if their own queue is empty, even if customers are waiting in another queue, thus sacrificing some efficiency. To analyze this tradeoff, we will study a parallel queueing system in which each arriving customer is immediately routed to the queue containing the smallest number of customers, namely the Join the Shortest Queue (JSQ) policy. We consider the system in (Halfin-Whitt) heavy traffic by allowing the arrival rate $\lambda_n$ to depend on $n$, letting the
quantity $(1-\lambda_n)\sqrt{n}$ have a non-degenerate limit, which we denote $\beta > 0$.
Note that JSQ is a logical first step for understanding the tradeoffs involved in maintaining parallel queues, because Winston \cite{winston} proved that among policies immediately assigning customers to one of $n < \infty$ parallel queues, JSQ is optimal in the case of Poisson arrivals and exponential service times. That is, it maximizes, with respect to stochastic order, the number of customers served in a given time interval. Weber \cite{weber} extended this result to the more general class of service times with non-decreasing hazard rate, with no assumptions on the arrival process. We will consider Poisson arrivals and exponential service times, and denote this system $M/M/n$-JSQ, distinguishing it from the traditional $M/M/n$ system which maintains a global queue.

Our main result describes the behavior of processes counting the number of idle servers and of queues with one customer waiting to enter service, along with auxiliary processes to count the number of longer queues. We prove that a system that initially has a fixed maximum queue length, appropriately scaled, converges weakly to a diffusion process as $n$ approaches infinity. The coordinates of this diffusion process representing queues with more than 1 customer in service are deterministic and show that any longer queues present in the initial condition disappear in fixed time and do not form again. The coordinates corresponding to the number of idle servers and number of queues with exactly one customer waiting correspond to a two-dimensional reflected \OU process. The entire limiting system will be defined in terms of a stochastic integral equation which we prove has a unique solution. This existence and uniqueness result is stated in Theorem \ref{thm:app-int} and the weak convergence result is stated in Theorem \ref{thm:trunc-result}, which is our main result.

As a proof technique, we will introduce a truncated variant of the $M/M/n$-JSQ system in which no queues of length longer than 2 are created, though such long queues are allowed in the initial condition. This system is more easily analyzed because it is finite dimensional and has limited interaction between many of the dimensions. In this truncated system, we show that the probability the system hits the truncation barrier decreases to zero as $n$ approaches infinity, and thus in the limit the behavior of the truncated and untruncated systems are the same. 

One consequence of our result is that both the number of idle servers and the number of queues with exactly one customer waiting in the $M/M/n$-JSQ system are of the order $\bo{\sqrt{n}}$.

Another feature of interest in this queueing system is the waiting time experienced by arriving customers. We prove that in the transient system the aggregate waiting time experienced by all customers is of the order $\bo{\sqrt{n}}$, and since the number of customers arriving in that time is order $n$, the waiting time per customer is $\bo{1/\sqrt{n}}$. We also observe that any arriving customer who has to wait will incur a waiting time which is exponentially distributed with parameter 1, which is the service time of the customer in service when they enter a queue. Since any waiting customers must incur a constant order waiting time and the aggregate waiting time is $\bo{\sqrt{n}}$, the fraction of customers who end up waiting is of the order $\bo{1/\sqrt{n}}$.

Next we review prior literature. The JSQ model was initially studied in the special case of 2 queues by Haight \cite{haight}. Kingman \cite{kingman} proved stability results along with considering the stationary distribution of the system, and Flatto and McKean \cite{flatto} also examine the stationary distribution. Further work on the $n = 2$ case includes bounds on the distribution of the number of people in the system by Halfin \cite{halfin-85}.

Foschini and Salz \cite{foschini} consider diffusion limits for the heavy traffic case of the $M/M/2$-JSQ system, first proving that the queue-length processes for the two queues are identical in the limit and then deriving the limiting distribution. The limiting behavior of the waiting time is the same as the standard $M/M/2$ system in heavy traffic. Their results extend to the case of $k$ parallel queues, but they do not consider the case where the number of queues grows as the traffic intensity increases. Zhang and Wang \cite{zhang-89} and Zhang and Hsu \cite{zhang-95} look at a similar problem but drop the assumption of Poisson arrivals and exponential service times, deriving functional central limit theorems for the heavy traffic JSQ system with $s$ servers.

Thus our paper is the first study of JSQ systems in the asymptotic regime as $n \to \infty$. Observe that for fixed $\lambda < 1$ as $n$ increases the probability of any customer arriving to find all servers busy will decrease to zero. In this case the JSQ nature of the system becomes irrelevant as customers will be assigned to an idle server immediately upon arrival. In particular we see that the limiting behavior of the system will essentially be that of the $M/M/\infty$ system, and thus it is of interest to consider this model in heavy traffic with $\lambda$ approaching unity.
There has been some work on models similar to ours, most notably Tezcan \cite{tezcan}, who considers a variant of the JSQ system with multiple pools of servers who each have their own queue. He uses a state-space collapse argument based on a framework of Dai and Tezcan \cite{dai-tezcan} to prove diffusion limits under the Halfin-Whitt heavy traffic regime. In this case that regime has the number of servers and traffic intensity increasing together in the limit, but the number of pools of servers is fixed so the number of queues is also fixed. Therefore our model is similar to Tezcan's but is not a special case of it. The state-space collapse argument implies that in the limit the system can be fully described by the total number of people in the system (rather than the queue lengths in the individual pools) and the diffusion limit of that process is very similar to the original Halfin and Whitt result \cite{halfin-whitt}.

Another branch of analysis of JSQ-like queueing systems has focused on the ``supermarket model'' in which arriving customers join the shortest queue from among $d$ randomly selected queues rather than from the entire system. It was proved independently by Mitzenmacher \cite{mitzenmacher} and Vvedenskaya, Dobrushin, and Karpelevich \cite{vvedenskaya} that this system achieves an exponential improvement in expected waiting time over a system with $n$ independent $M/M/1$ queues. Versions of this system where $d$ depends on $n$ are particularly closely related to our JSQ model, which essentially sets $d = n$. Brightwell and Luczak \cite{luczak} give a set of $d$ and $\lambda$ values depending on $n$ for which they prove the steady-state system is usually in a particular state with most queues having the same (known) length. Their conditions require $(1 - \lambda)^{-1} > d$, which excludes the $d = n$, $(1 - \lambda)\sqrt{n} \to \beta$ case considered in this paper. Dieker and Suk \cite{dieker-suk} prove fluid and diffusion limits for queue length processes when when $d$ increases to infinity at a rate slower than $n$ and with fixed $\lambda < 1$.

The remainder of the paper is laid out as follows: Section \ref{sec:model} defines the model and states our main result. In Section \ref{sec:int-rep} we will prove Theorem \ref{thm:app-int}, verifying that the integral representation of the limiting system is well defined. This result will also be the key to proving convergence via a continuous mapping theorem (CMT) argument. Section \ref{sec:mart-rep} will construct a representation of the system as a combination of martingales and reflecting processes. In Section \ref{sec:mart-conv} we will establish the convergence properties of these martingales, and then apply the CMT to translate the convergence of martingales to the convergence of the scaled queue length processes. This section will conclude our proof of Theorem \ref{thm:trunc-result}. In Section \ref{sec:wait} we discuss the waiting time in the $M/M/n$-JSQ system. We will conclude in Section \ref{sec:discuss} with a brief discussion of the implications of Theorem \ref{thm:trunc-result} and possible extensions.

We use $\Rightarrow$ to denote weak convergence, $\1\{A\}$ to denote the indicator function for the event $A$, $(x)^+ = \max(x,0)$. We let $\Rbar_+ = \R_+ \cup \{\infty\}$ represent the extended positive real line. We will equip $\Rbar_+$ with the order topology, in which neighborhoods of $\infty$ are those sets which contain a subset of the form $\{x > a\}$ for some $a \in \R$. Most processes in this paper will live in the space $D = D([0,\infty),\R)$ of right continuous functions with left limits mapping $[0,\infty)$ into $\R$. We also consider $D^k = D([0,\infty),\R^k)$ for $k \ge 2$, which we will treat as the product space $D \times D \times \cdots \times D$ (see, e.g., \cite{whitt-stoch-proc} \S 3.3). We will denote the uniform norm
\[\vn{x}_t = \sup_{0\le s \le t}|x(s)|\]
for $x \in D$ and the max norm
\[\vn{(x_1,\ldots,x_k)}_t = \max_{1 \le i \le k}\vn{x_i}_t\]
for $x \in D^k$. Similarly we will use the max norm
\[|b| = \max_{1 \le i \le k} |b_i|\]
for $b \in \R^k$.

\section{The model and the main result.}\label{sec:model}

We consider a $M/M/n$-JSQ queueing system with $n$ servers where each server maintains a unique queue, with service proceeding according to the first-in-first-out discipline. Service time is exponentially distributed at each server, with the rate fixed at 1. Arrivals occur in a single stream, as a Poisson process with rate $\lambda_n n$, where $0 < \lambda_n < 1$ and
\begin{equation}\label{eq:lambda-scale}
\lim_{n \to \infty} \sqrt{n}(1 - \lambda_n) = \beta
\end{equation}
for fixed $\beta > 0$. Upon arrival, each customer is routed to the server with the shortest queue. In the event of a tie, one of the options is selected uniformly at random.

The state of the system will be represented via the process $Q^n(t) = (Q_1^n(t), Q_2^n(t),\ldots)$, with $Q_i^n(t)$ representing the number of queues with \emph{at least} $i$ customers (including any customer in service) at time $t \ge 0$. We note that for the system as described we have
\begin{equation}\label{eq:monotonicity}
n \ge Q_1^n(t) \ge Q_2^n(t) \ge \cdots \ge 0, \quad\quad \forall t \ge 0,
\end{equation}
and that we can recover the number of queues with exactly $i$ customers in service via the quantity $Q_{i}^n(t) - Q_{i+1}^n(t)$, including the number of idle servers $n - Q_1^n(t)$.

To state our weak convergence results, we also introduce a scaled version $X^n(t)$ of this process defined as
\begin{equation}\label{eq:X-def}
X_1^n(t) = \frac{Q_1^n(t) - n}{\sqrt{n}} \quad \text{ and } \quad X_i^n(t) = \frac{Q_i^n(t)}{\sqrt{n}}, \quad i \ge 2.
\end{equation}
The $i = 1$ case is treated differently because the number of queues with length 1 behaves differently than the number of queues of all larger lengths. In particular, there will be $\bo{\sqrt{n}}$ idle servers, and thus the number of servers with at least one customer in service will be order $n$.

Our diffusion limit will be the solution to a system of $k$ integral equations for some $k \ge 3$, so we first introduce this system and prove that it has a unique solution. Furthermore, we prove that the system defines a continuous map from $\Rbar_+ \times \R^k \times D^k$ to $D^k \times D^2$ with respect to appropriate topologies. This continuity, along with the further fact that the function maps continuous functions to continuous functions, allows us to use the CMT to prove weak convergence once we show the weak convergence of the arguments.

\begin{theorem}\label{thm:app-int}
Given integer $k \ge 3$, $B \in \Rbar_+$, $b \in \R^k$, and $y \in D^k$, consider the following system:
\begin{align}
x_1(t) &= b_1 + y_1(t) + \int_0^t(-x_1(s) + x_2(s))ds - u_1(t), \label{eq:int-app-1}\\
x_2(t) &= b_2 + y_2(t) + \int_0^t(-x_2(s) + x_3(s))ds + u_1(t) - u_2(t), \label{eq:int-app-2}\\
x_i(t) &= b_i + y_i(t) + \int_0^t(-x_i(s) + x_{i+1}(s))ds,\quad 3 \le i \le k-1, \\
x_k(t) &= b_k + y_k(t) + \int_0^t-x_k(s)ds, \\
x_1(t) &\le 0, \quad 0 \le x_2(t) \le B, \quad x_i(t) \ge 0, \quad t \ge 0, \label{eq:int-app-3}
\end{align}
with $u_1$ and $u_2$ nondecreasing nonnegative functions in $D$ such that
\begin{align*}
\int_0^\infty \1\{x_1(t) < 0\}du_1(t) &= 0, \\
\int_0^\infty \1\{x_2(t) < B\}du_2(t) &= 0.
\end{align*}
Then \eqref{eq:int-app-1}-\eqref{eq:int-app-3} has a unique solution $(x,u) \in D^k \times D^2$ so that there is a well defined function $(f,g):\Rbar_+ \times \R^k \times D^k\to D^k \times D^2$ mapping $(B,b,y)$ into $x = f(B,b,y)$ and $u = g(B,b,y)$.
Furthermore, the function $(f,g)$ is continuous on $\Rbar_+ \times \R^k \times D^k$ with respect to the product topology when $\Rbar_+$ is equipped with the order topology and $D$ is equipped with the topology of uniform convergence over bounded intervals. Finally, if $y$ is continuous, then so are $x$ and $u$.
\end{theorem}

We will prove this theorem in Section \ref{sec:int-rep}. One implication of Theorem \ref{thm:app-int} is that the limiting system we find in our main result below is well defined because, as we will see, it is an application of the function $(f,g)$ with specific arguments $b,y$, and $B = \infty$ augmented with $X_i(t) = 0$ for $i > k$. Note that $B = \infty$ implies $u_2 = 0$. Our main result is the following:

\begin{theorem} \label{thm:trunc-result}
In the sequence of $M/M/n$-JSQ models described above, suppose there exists $k$ and a random vector $X(0) = (X_1(0), \ldots, X_k(0)) \in \R^k$ such that
\begin{equation}\label{eq:trunc-initial}
X_i^n(0) \Rightarrow X_i(0) \quad \text{ in } \R \text{ as } n \to \infty, \quad 1 \le i \le k,
\end{equation}
and $X_i^n(0) = 0$ for $i > k$.
Then for any $t \ge 0$,
\[X_i^n \Rightarrow X_i \quad \text{ in } D \text{ as } n \to \infty, \quad i \ge 1,\]
where $X_1\le 0$ and $X_i\ge 0$ for $i \ge 2$ are unique solutions in $D$ of the stochastic integral equations
\begin{align}
X_1(t) &= X_1(0) + \sqrt{2}W(t) - \beta t + \int_0^t\left(-X_1(s) + X_2(s)\right)ds - U_1(t), \label{eq:trunc-limit-1}\\
X_2(t) &= X_2(0) + U_1(t) + \int_0^t(-X_2(s) + X_3(s))ds, \label{eq:trunc-limit-2} \\
X_i(t) &= X_i(0) + \int_0^t(-X_i(s) + X_{i+1}(s))ds, \quad 3 \le i \le k - 1, \label{eq:trunc-limit-i}\\
X_k(t) &= X_k(0) + \int_0^t-X_k(s)ds, \label{eq:trunc-limit-k}\\
X_i(t) &= 0, \quad i \ge k + 1,
\end{align}
where $W$ is a standard Brownian motion and $U_1$ is the unique nondecreasing nonnegative process in $D$ satisfying
\begin{align}
\int_0^\infty\1\{X_1(t) < 0\}dU_1(t) &= 0. \label{eq:trunc-reflection}
\end{align}

\end{theorem}

\begin{remark}
The integral equations \eqref{eq:trunc-limit-i}-\eqref{eq:trunc-limit-k} are deterministic and have an explicit solution:
\begin{align*}
X_i(t) &= e^{-t}\left(X_i(0) + \sum_{j = 1}^{k-i}\frac{1}{j!}t^jX_{i+j}(0)\right), \quad 3 \le i \le k - 1, \\
X_k(t) &= X_k(0) e^{-t}.
\end{align*}
Thus the number of queues of length at least $i$ for $i \ge 3$ decays exponentially in time.
\end{remark}

We note that condition \eqref{eq:trunc-initial} does place significant but not unreasonable restrictions on the starting state of the finite systems $Q^n$. In particular, $Q_1^n(0) - n = \bo{\sqrt{n}}$ so the number of customers initially in service must be sufficiently near $n$. Similarly, \eqref{eq:trunc-initial} requires $Q_2^n(0) = \bo{\sqrt{n}}$ and therefore $Q_i^n(0) = \bo{\sqrt{n}}$ for $3 \le i \le k$. Also note that the longest queue allowed in the initial condition has length $k$.

Our result shows that the $M/M/n$-JSQ system in the heavy traffic limit becomes essentially a two-dimensional system. If queues with more than one customer waiting are present initially, they disappear and do not form again. There are $\bo{\sqrt{n}}$ idle servers and $\bo{\sqrt{n}}$ queues with exactly one customer, and the behavior of processes counting these correspond to a two-dimensional \OU process.

\section{Integral representation.}\label{sec:int-rep}
We will now prove Theorem \ref{thm:app-int}, showing that the representation of the limiting system in Theorem \ref{thm:trunc-result} is a valid and unique representation. We will also show that it defines a continuous map from $\Rbar_+ \times \R^k \times D^k$ to $D^k\times D^2$. The continuity of the map in the topology of uniform convergence over bounded intervals will allow us to use the continuous mapping theorem (CMT) to demonstrate the convergence $X^n_i \Rightarrow X_i$ once we write $X^n_i$ in the appropriate integral form.

Note that by using $\Rbar_+$ in the domain of this map we allow the upper barrier $B$ for the function $x_2$ to take the value $\infty$, which corresponds to there being no upper barrier on the $\sqrt{n}$ scale.

Our approach to proving Theorem \ref{thm:app-int} will involve two main lemmas, one dealing with the first two dimensions, where reflection plays an important role, and a one dimensional lemma that we will apply to the higher dimensions.

\subsection{Lower dimensions.}
To deal with the reflection terms in the first two dimensions of Theorem \ref{thm:app-int} it is convenient to consider those dimensions completely decoupled from the rest of the system. To that end we will prove the following:
\begin{lemma}\label{thm:trunc-int-low}
Given $B \in \Rbar_+$, $b \in \R^2$, and $y \in D^2$, consider
\begin{align}
x_1(t) &= b_1 + y_1(t) + \int_0^t(-x_1(s) + x_2(s))ds - u_1(t), \label{eq:int-trunc-1}\\
x_2(t) &= b_2 + y_2(t) + \int_0^t(-x_2(s))ds + u_1(t) - u_2(t), \label{eq:int-trunc-2}\\
x_1(t) &\le 0, \quad 0 \le x_2(t) \le B, \quad t \ge 0, \label{eq:int-trunc-3}
\end{align}
with $u_1$ and $u_2$ nondecreasing nonnegative functions in $D$ such that
\begin{align*}
\int_0^\infty \1\{x_1(t) < 0\}du_1(t) &= 0, \\
\int_0^\infty \1\{x_2(t) < B\}du_2(t) &= 0.
\end{align*}
Then \eqref{eq:int-trunc-1}-\eqref{eq:int-trunc-3} has a unique solution $(x,u) \in D^2 \times D^2$ so that there is a well defined function $(f,g):\Rbar_+ \times \R^2 \times D^2\to D^2 \times D^2$ mapping $(B,b,y)$ into $x = f(B,b,y)$ and $u = g(B,b,y)$.
Furthermore, the function $(f,g)$ is continuous. Finally, if $y$ is continuous, then so are $x$ and $u$.
%has a unique solution which is continuous and preserves continuity, as in Theorem \ref{thm:app-int}.
\end{lemma}

\subsubsection{The reflection map.}\label{sec:ref-map}
In several places we will make use of the well known one-dimensional reflection map for an upper barrier. Given upper barrier $\kappa \in \R_+$, we let $(\phi_\kappa, \psi_\kappa) : D \to D^2$ be the one-sided reflection map with upper barrier at $\kappa$ (see, e.g., \cite{whitt-stoch-proc} \S 5.2 and \S 13.5). In particular for $x \in D$ with $x(0) \le \kappa$ we have $z = \psi_{\kappa}(y) \ge 0$, $z$ nondecreasing,
\[x = \phi_{\kappa}(y) = y - z \le \kappa,\]
and
\[\int_0^\infty \1\{x < \kappa\} dz = 0.\]
Recall that these functions can be defined explicitly by
\begin{equation}\label{eq:regulator}
\psi_\kappa(x)(t) = \sup_{0\le s \le t}(x(s) - \kappa)^+
\end{equation}
and
\begin{equation}\label{eq:reflected}
\phi_\kappa(x)(t) = x(t) - \psi_\kappa(x)(t).
\end{equation}
We will also make use of a slight variant of the usual Lipschitz condition for these functions to allow for different values of $\kappa$. In particular, for $x,x' \in D$, $\kappa, \kappa' \in \R$, and $t \ge 0$ we have
\begin{align}
\vn{\psi_{\kappa}(x)-\psi_{\kappa'}(x')}_t \le \vn{x-x'}_t + |\kappa-\kappa'|, \label{eq:lips-reg} \\
\vn{\phi_{\kappa}(x)-\phi_{\kappa'}(x')}_t \le 2\vn{x-x'}_t + |\kappa-\kappa'|. \label{eq:lips-ref}
\end{align}
These follow straightforwardly from \eqref{eq:regulator} and \eqref{eq:reflected}. Note that for $\kappa = \kappa'$ we recover the usual Lipschitz constants of $1$ for $\psi_\kappa$ and $2$ for $\phi_\kappa$.

We also define a trivial reflection map for $\kappa = \infty$ by letting $(\phi_\infty,\psi_\infty) = (e,0)$ where $e$ is the identity map. That is, the reflection map leaves the argument unchanged and the regulator is identically zero. We prove the following:

\begin{lemma}\label{thm:reflection}
The function $(\phi,\psi) : \Rbar_+ \times D \to D^2$ defined by \eqref{eq:regulator}-\eqref{eq:reflected} for finite $\kappa$ and by $(\phi_\infty,\psi_\infty) = (e,0)$ for $\kappa = \infty$ is continuous with respect to the product topology when $\Rbar_+$ is equipped with the order topology and $D$ is equipped with the topology of uniform convergence over bounded intervals.
\end{lemma}

\begin{proof}
By \eqref{eq:lips-reg}-\eqref{eq:lips-ref} the function is continuous at any finite $\kappa \in \R_+$. For $x \in D$ and $x^\kappa \in D$ such that $x^\kappa \to x$ as $\kappa \to \infty$,
\begin{align*}
\lim_{\kappa \to \infty}\vn{\psi_{\kappa}(x^\kappa)}_t &= \lim_{\kappa \to \infty}\sup_{0 \le s \le t}|\psi_{\kappa}(x^\kappa)| \\
&= \lim_{\kappa \to \infty}\sup_{0 \le s \le t}(x^\kappa(s) - \kappa)^+ \\
&= \sup_{0 \le s \le t}\lim_{\kappa \to \infty}(x^\kappa(s) - \kappa)^+ \\
&= 0,
\end{align*}
where we have made use of the fact that $\vn{x}_t < \infty$. Therefore $\psi_\kappa(x^\kappa) \to \psi_{\infty}(x)$ and by \eqref{eq:reflected} we conclude $\phi_{\kappa}(x^\kappa) \to \phi_\infty(x)$. Thus the function is continuous at $\kappa = \infty$, completing the proof. 
\end{proof}

With these facts about the reflection map in hand, we will now prove a result similar to Lemma \ref{thm:trunc-int-low} for a related system:
\begin{lemma}\label{thm:app-int-unrefl}
Given $B \in \Rbar_+$, $b \in \R^2$ and $y \in D^2$, consider
\begin{align}
w_1(t) &= b_1 + y_1(t) + \int_0^t\left(-\phi_0(w_1(s)) + \phi_B(w_2(s))\right)ds, \label{eq:w-1}\\
w_2(t) &= b_2 + y_2(t) + \psi_0(w_1(t)) + \int_0^t\left(-\phi_B(w_2(s))\right)ds \ge 0. \label{eq:w-2}
\end{align}
Then \eqref{eq:w-1}-\eqref{eq:w-2} has a unique solution %which is continuous and preserves continuity, as in Theorem \ref{thm:app-int}.
$w \in D^2$ so that there is a well defined function $\xi:\Rbar_+ \times \R^2 \times D^2 \to D^2$ mapping $(B,b,y)$ into $w = \xi(B,b,y)$.  Furthermore, the function $\xi$ is continuous. Finally, if $y$ is continuous, then so is $w$.
\end{lemma}

Before proceeding with the proof we introduce a version of Gronwall's inequality first proved by Greene \cite{greene} and proved in the form we use by Das \cite{das}:

\begin{lemma}[Gronwall's inequality]\label{thm:gronwall}
Let $K_1$ and $K_2$ be nonnegative constants, let $h_i$ be real constants, and let $f,g$ be continuous nonnegative functions for all $t \ge 0$ such that
\begin{align*}
f(t) &\le K_1 + h_1\int_0^tf(s)ds + h_2\int_0^tg(s)ds, \\
g(t) &\le K_2 + h_3\int_0^tf(s)ds + h_4\int_0^tg(s)ds
\end{align*}
for all $t \ge 0$. Then
\[f(t) \le Me^{ht} \quad\text{ and }\quad g(t) \le Me^{ht}\]
for all $t \ge 0$ where $M = K_1 + K_2$ and $h = \max\{h_1+h_3,h_2+h_4\}$. In particular, if $K_1,K_2 = 0$, then $f(t),g(t) = 0$ for all $t$.
\end{lemma}

\begin{proof}[Proof of Lemma \ref{thm:app-int-unrefl}.]
We will show existence via a contraction mapping argument. First we will show that for $t \ge 0$ there exists a solution $\tilde{w} = (\tilde{w}_1,\tilde{w}_2)$ to the system of integral equations
\begin{align}
\tilde{w}_1(t) &= b_1 + y_1(t) + \int_0^t\left(-\phi_0(\tilde{w}_1(s)) + \phi_B\big(\tilde{w}_2(s) + \psi_0(\tilde{w}_1(s))\big)\right)ds, \label{eq:alt-w-1}\\
\tilde{w}_2(t) &= b_2 + y_2(t) + \int_0^t\left(-\phi_B\big(\tilde{w}_2(s)+\psi_0(\tilde{w}_1(s))\big)\right)ds \ge 0. \label{eq:alt-w-2}
\end{align}
Once we have such a solution, it follows immediately that
\[w = (w_1,w_2) = \left(\tilde{w}_1, \tilde{w}_2 + \psi_0(\tilde{w}_1)\right)\]
is a solution to \eqref{eq:w-1}-\eqref{eq:w-2}.

We first show that the map defined by the right hand side of \eqref{eq:alt-w-1}-\eqref{eq:alt-w-2} is a contraction for small enough $t$. We define $T: D^2 \to D^2$ by
\begin{align}
T(\tilde{w})_1(t) &= b_1 + y_1(t) + \int_0^t\left(-\phi_0(\tilde{w}_1(s)) + \phi_B\big(\tilde{w}_2(s) + \psi_0(\tilde{w}_1(s))\big)\right)ds, \label{eq:T-1}\\
T(\tilde{w})_2(t) &= b_2 + y_2(t) + \int_0^t\left(-\phi_B\big(\tilde{w}_2(s)+\psi_0(\tilde{w}_1(s))\big)\right)ds. \label{eq:T-2}
\end{align}
For $\tilde{w},\tilde{v} \in D^2$ we have
\begin{align*}
\vn{T(\tilde{w})_1-T(\tilde{v})_1}_t &\le \int_0^t\vn{-\phi_0(\tilde{w}_1) +\phi_0(\tilde{v}_1)}_sds \\
&\quad\quad + \int_0^t\vn{\phi_B(\tilde{w}_2+\psi_0(\tilde{w}_1)) - \phi_B(\tilde{v}_2 + \psi_0(\tilde{v}_1))}_sds \\
&\le 2\int_0^t\vn{\tilde{w}_1 - \tilde{v}_1}_sds \\
&\quad\quad+ 2\int_0^t\vn{\tilde{w}_2 +\psi_0(\tilde{w}_1) - \tilde{v}_2 - \psi_0(\tilde{v}_1)}_sds \\
&\le 2t\vn{\tilde{w}_1 - \tilde{v}_1}_t + 2t\vn{\tilde{w}_2 - \tilde{v}_2}_t + \int_0^t\vn{\tilde{w}_1 - \tilde{v}_1}_sds \\
&\le 2t\vn{\tilde{w}_1 - \tilde{v}_1}_t + 2t\vn{\tilde{w}_2 - \tilde{v}_2}_t + t\vn{\tilde{w}_1 - \tilde{v}_1}_t \\
&\le 5t\vn{\tilde{w}-\tilde{v}}_t
\end{align*}
and
\begin{align*}
\vn{\tilde{w}_2 - \tilde{v}_2}_t &\le \int_0^t\vn{-\phi_B(\tilde{w}_2+\psi_0(\tilde{w}_1)) + \phi_B(\tilde{v}_2 + \psi_0(\tilde{v}_1))}_sds \\
&\le 2t\vn{\tilde{w}_2 + \psi_0(\tilde{w}_1)-\tilde{v}_2 - \psi_0(\tilde{v}_1)}_t \\
&\le 2t\vn{\tilde{w}_2-\tilde{v}_2}_t + t\vn{\tilde{w}_1-\tilde{v}_1}_t \\
&\le 3t\vn{\tilde{w} - \tilde{v}}_t.
\end{align*}
We therefore conclude that
\[\vn{T(\tilde{w}) - T(\tilde{v})}_t \le 5t \vn{\tilde{w}-\tilde{v}}_t,\]
so for $t_0 < \frac{1}{5}$, $T$ is a contraction on $D([0,t_0],\R^2)$. Therefore by the contraction mapping principle (see, e.g., \cite[p.220]{rudin}), $T$ has a unique fixed point $\tilde{w}$ on $D([0,t_0],\R^2)$ such that $T(\tilde{w}) = \tilde{w}$. This fixed point solves \eqref{eq:alt-w-1}-\eqref{eq:alt-w-2} for $t \in [0,t_0]$. Now we extend the fixed point argument to $t \in [t_0,2t_0], [2t_0,3t_0], \ldots$ and repeat to find a solution $\tilde{w}$ to \eqref{eq:alt-w-1}-\eqref{eq:alt-w-2} for $t \ge 0$. As noted above, this provides a solution $w$ to \eqref{eq:w-1}-\eqref{eq:w-2}.

To prove uniqueness of this solution, suppose $w$ and $w'$ are two solutions to \eqref{eq:w-1}-\eqref{eq:w-2}. We consider
\begin{align}
\vn{w_1-w_1'}_t &\le \int_0^t\vn{-\phi_0(w_1) + \phi_0(w_1') +\phi_B(w_2)-\phi_B(w_2')}_sds \notag \\
&\le 2 \int_0^t\left(\vn{w_1 - w_1'}_s + \vn{w_2 - w_2'}_s\right)ds \label{eq:uniq-bound-1}
\end{align}
and
\begin{align}
\vn{w_2-w_2'}_t &\le \vn{\psi_0(w_1) - \psi_0(w_1')}_t + \int_0^t\vn{\phi_B(w_2)-\phi_B(w_2')}_sds  \notag\\
&\le \vn{w_1 - w_1'}_t + 2\int_0^t\vn{w_2-w_2'}_sds. \label{eq:uniq-bound-2}
\end{align}
To match the form of Gronwall's inequality (Lemma \ref{thm:gronwall})  we rewrite \eqref{eq:uniq-bound-2} as
\[\vn{w_2-w_2'}_t - \vn{w_1-w_1'}_t \le 2\int_0^t\vn{w_2-w_2'}_sds\]
and note that the right hand side is nonnegative so the inequality remains true as
\begin{equation}\label{eq:uniq-bound-3}
\left(\vn{w_2-w_2'}_t - \vn{w_1-w_1'}_t\right)^+ \le 2\int_0^t\vn{w_2-w_2'}_sds.
\end{equation}

We now define
\begin{align*}
u_1(t) &= \vn{w_1 - w_1'}_t, \\
u_2(t) &= \left(\vn{w_2-w_2'}_t - \vn{w_1-w_1'}_t\right)^+
\end{align*}
and note
\begin{equation}\label{eq:uniq-aux}
\vn{w_2-w_2'}_s \le u_2(s) + u_1(s) \quad s \ge 0.
\end{equation}
Then \eqref{eq:uniq-bound-1}, \eqref{eq:uniq-bound-3}, and \eqref{eq:uniq-aux} imply
\begin{align*}
u_1(t) \le 4\int_0^tu_1(s)ds + 2\int_0^tu_2(s)ds, \\
u_2(t) \le 2\int_0^tu_1(s)ds + 2\int_0^tu_2(s)ds.
\end{align*}
Now by Gronwall's inequality we have
\begin{align*}
u_1(t) = 0 \quad \text{ and } \quad u_2(t) = 0,
\end{align*}
so by the definition of $u_1$ and \eqref{eq:uniq-aux} we have
\[\vn{w_1-w_1'}_t = \vn{w_2-w_2'}_t = 0\]
for all $t \ge 0$ and therefore the solution $w$ is unique.

We now establish the continuity of $\xi$. Suppose 
\[(B^n,b^n,y^n) \to (B,b,y) \quad \text{ as } n \to \infty.\]
Fix $\epsilon > 0$ and suppose $w^n$ and $w$ satisfy \eqref{eq:w-1}-\eqref{eq:w-2} for $(B^n,b^n,y^n)$ and $(B,b,y)$, respectively. Choose $N$ such that for all $n \ge N$,
\[|b^n - b| + \vn{y^n - y}_t + \vn{\phi_{B^n}(w_2) - \phi_{B}(w_2)}_t < \delta\]
for some $\delta > 0$ which is yet to be determined. Note that such an $N$ exists by Lemma \ref{thm:reflection} and the assumption $B^n \to B$. We have
\begin{align}
\vn{w^n_1-w_1}_t &\le |b^n-b| + \vn{y^n-y}_t \notag\\
&+ \int_0^t\vn{-\phi_0(w_1^n) + \phi_0(w_1) +\phi_{B^n}(w_2^n)-\phi_{B}(w_2)}_sds \notag \\
&\le \delta + \int_0^t\big(2\vn{w_1^n-w_1}_s + \vn{\phi_{B^n}(w_2^n)-\phi_{B^n}(w_2)}_s  \notag\\
&\quad\quad\quad\quad\quad + \vn{\phi_{B^n}(w_2) - \phi_B(w_2)}_s\big) ds \notag\\
&\le \delta + \int_0^t\left(2\vn{w_1^n - w_1}_s + 2\vn{w_2^n - w_2}_s + \delta\right)ds  \notag\\
&\le \delta(1+t) + 2 \int_0^t\left(\vn{w_1^n - w_1}_s + \vn{w_2^n - w_2}_s\right)ds \label{eq:unif-bound-1}
\end{align}
and
\begin{align}
\vn{w_2^n-w_2}_t \le \delta(1+t) + \vn{w_1^n-w_1}_t + 2\int_0^t\vn{w_2^n-w_2}_sds \label{eq:unif-bound-2}.
\end{align}

As in the uniqueness argument above, we will apply Gronwall's inequality, with functions
\[u_1(t) = \vn{w_1^n-w_1}_t \quad\text{ and }\quad u_2(t) = \left(\vn{w_2^n-w_2}_t - \vn{w_1^n-w_1}_t\right)^+.\]
Then we have
\begin{align*}
u_1(t) \le \delta(1+t) + 4\int_0^tu_1(s)ds + 2\int_0^tu_2(s)ds, \\
u_2(t) \le \delta(1+t) + 2\int_0^tu_1(s)ds + 2\int_0^tu_2(s)ds,
\end{align*}
so Gronwall's inequality implies
\[u_1(t) \le 2\delta(1+t) e^{6t} \quad \text{ and } \quad u_2(t) \le 2\delta(1+t) e^{6t}\]
and we have
\[\vn{w_1^n - w_1}_t \le 2\delta(1+t) e^{6t} \quad \text{ and } \quad \vn{w_2^n - w_2}_t \le 4\delta(1+t) e^{6t}.\]
We choose $\delta = \frac{1}{4+4t} \epsilon e^{-6t}$ to establish the desired continuity.

For the proof of continuity of $w$ we note
\[|w_1(t + s) - w_1(t)| \le |y_1(t+s) - y_1(t)| + \int_t^{t+s}|\phi_B(w_2(z)) - \phi_0(w_1(z))|dz\]
and
\[|w_2(t + s) - w_2(t)| \le |y_2(t+s) - y_2(t)| + |w_1(t+s) - w_1(t)| + \int_t^{t+s}|\phi_B(w_2(z))|dz\]
The boundedness of $w_1$ and $w_2$ proved in Lemma \ref{thm:stoch-bound} imply that $w_1$ and $w_2$ are continuous if $y_1$ and $y_2$ are continuous. 
\end{proof}

We are now prepared to prove Lemma \ref{thm:trunc-int-low}.

\begin{proof}[Proof of Lemma \ref{thm:trunc-int-low}.]
Our key insight is to see that a solution is found by setting $x_1 = \phi_0(w_1)$, $u_1 = \psi_0(w_1)$, $x_2 = \phi_B(w_2)$, and $u_2 = \psi_B(w_2)$  where $(w_1,w_2)$ is the unique solution defined by Lemma \ref{thm:app-int-unrefl}.

To see that it is unique, note that the conditions on $u_1$ and $u_2$ imply that they can be written as $\psi_0(z_1)$ and $\psi_B(z_2)$ for some functions $z_1,z_2 \in D$. Then $x_1$ and $x_2$ are $\phi_0(z_1)$ and $\phi_B(z_2)$ for the same $z_1$ and $z_2$. Then \eqref{eq:int-app-1}-\eqref{eq:int-app-3} imply that $z = (z_1,z_2)$ must be a solution of \eqref{eq:w-1}-\eqref{eq:w-2}. By Lemma \ref{thm:app-int-unrefl} this solution is unique. In particular, this solution is
\begin{align*}
x_1 = f_1(b,y) &= (\phi_0 \circ \xi_1)(b,y), \\
u_1 = g_1(b,y) &= (\psi_0 \circ \xi_1)(b,y), \\
x_2 = f_2(b,y) &= (\phi_B \circ \xi_2)(b,y), \\
u_2 = g_2(b,y) &= (\psi_B \circ \xi_2)(b,y).
\end{align*}
The reflection maps $(\phi_0,\psi_0)$ and $(\phi_B,\psi_B)$ are continuous in the uniform topology and also preserve continuity. Since $\xi$ also has these properties by Lemma \ref{thm:app-int-unrefl}, we conclude that $(f,g)$ are continuous and preserve continuity. 
\end{proof}

\subsection{Higher dimensions.}
Because of the lack of reflection terms for $i \ge 3$, the higher dimensional terms in Theorem \ref{thm:app-int} behave in primarily one-dimensional ways. Therefore we will note an existence, uniqueness, and continuity result for a one-dimensional system that we will use in our proof of Theorem \ref{thm:app-int}. This lemma is a special case of a slightly more general result proved by Pang, Talreja, and Whitt \cite[Theorem 4.1]{martingale}, namely letting $h(x) = -x$.

\begin{lemma}\label{thm:int-one}
Given $b \in \R$, and $y \in D$, consider
\begin{align}
x(t) &= b + y(t) + \int_0^t-x(s)ds \label{eq:int-one-1}\\
x(t) &\ge 0, \quad t \ge 0 \label{eq:int-one-2}
\end{align}

Then \eqref{eq:int-one-1}-\eqref{eq:int-one-2} has a unique solution $x \in D$ so that there is a well defined function $f:\R \times D \to D$ mapping $(b,y)$ into $x = f(b,y)$. Furthermore, the function $(f,g)$ is continuous. Finally, if $y$ is continuous, then so is $x$.
\end{lemma}

With this lemma in hand, we can prove Theorem \ref{thm:app-int}:

\begin{proof}[Proof of Theorem \ref{thm:app-int}.]
We first prove existence, uniqueness, and preservation of continuity by induction on $k \ge i \ge 3$.

As the base case, observe that Lemma \ref{thm:int-one} implies there exists a unique solution $x_k(t)$ which preserves continuity of $y_k(t)$.

As an induction hypothesis, suppose for some $i \ge 3$ there exist unique solutions $x_{i+1}(t), \ldots, x_{k}(t)$ which are continuous if $y_{i+1}, \ldots, y_k$ are continuous, and consider $x_i(t)$.

We have the integral equation
\begin{align*}
x_i(t) &= b_i + y_i(t) + \int_0^t(-x_i(s)+x_{i+1}(s))ds \\
 &= b_i + y_i(t) + \int_0^tx_{i+1}(s)ds + \int_0^t-x_i(s)ds.
\end{align*}
By the induction hypothesis $x_{i+1}(t)$ exists and is unique, so we let
\[\hat{y}(t) = y_i(t) + \int_0^tx_{i+1}(s)ds\]
and apply Lemma \ref{thm:int-one} with $y = \hat{y}$ to conclude that there exists a unique solution $x_i(t)$. This solution is continuous if $y_i, \ldots, y_k$ are continuous because in that case that $\hat{y}(t)$ is continuous.

By induction, we conclude that there exist unique solutions $x_3, \ldots, x_k$ which preserve continuity of $y_3, \ldots, y_k$. From this, we can define
\[\hat{y}_2(t) = y_2(t) + \int_0^tx_{3}(s)ds\]
and apply Lemma \ref{thm:trunc-int-low} with $y_2 = \hat{y}_2$ to complete the proof of existence, uniqueness, and preservation of continuity.

To verify that the map $(f,g)$ is continuous, suppose $x^n(t)$ and $x(t)$ solve \eqref{eq:int-app-1}-\eqref{eq:int-app-3} for $(B^n,b^n,y^n)$ and $(B,b,y)$, respectively, and further suppose $(B^n,b^n,y^n) \to (B,b,y)$. 

We again proceed by induction on $k \ge i \ge 3$.

For the base case note Lemma \ref{thm:int-one} implies $x_k^n \to x_k$.

As an induction hypothesis, suppose for some $i \ge 3$ we have $x^n_{j} \to x_j$ for $i \le j \le k$.

By the induction hypothesis we have
\[y^n_{i}(t) + \int_0^tx^n_{i+1}(s)ds \to y_{i}(t) + \int_0^tx_{i+1}(s)ds,\]
and thus, again by Lemma \ref{thm:int-one}, $x_{i}^n \to x_{i}$. 

By induction we conclude $x_i^n \to x_i$ for $3 \le i \le n$. This further implies
\[y^n_{2}(t) + \int_0^tx^n_{3}(s)ds \to y_{2}(t) + \int_0^tx_{3}(s)ds,\]
so Lemma \ref{thm:trunc-int-low} implies $(x_1^n,x_2^n) \to (x_1,x_2)$, and continuity is established.
\end{proof}

\section{Truncation and martingale representation.}\label{sec:mart-rep}
To use Theorem \ref{thm:app-int} in a CMT argument, we want to write the process $X^n$ in the appropriate integral form. Instead of directly considering the full $M/M/n$-JSQ system, however, we will introduce a truncated variant which we will later show has the same behavior as $X^n$ in the limit. It is this truncated system that will be shown to take the integral form of Theorem \ref{thm:app-int}. 

\subsection{Truncation.}
An important feature of the behavior of the $M/M/n$-JSQ system is that queues with more than one customer waiting are formed only if all queues have at least one customer waiting already. One of our goals is to show that $Q_2^n$, the number of queues with a customer waiting, is of the order $\bo{\sqrt{n}}$ and thus it is unlikely for longer queues to form. With this in mind, as a proof technique, we now introduce a truncated version of the system in which no queues with length greater than 2 are created, though they are allowed to exist in the initial condition. This system will be significantly easier to analyze and will be shown to have stochastic behavior exactly matching the untruncated system with high probability. See Figure \ref{fig:full-start} to see a sample path of an original untruncated system which starts with order $\bth{\sqrt{n}}$ queues of length 4. Note that the number of queues length 3 and length 4 decrease monotonically.

\begin{figure}[h]
\begin{subfigure}{0.5\textwidth}
\includegraphics[scale = 0.4]{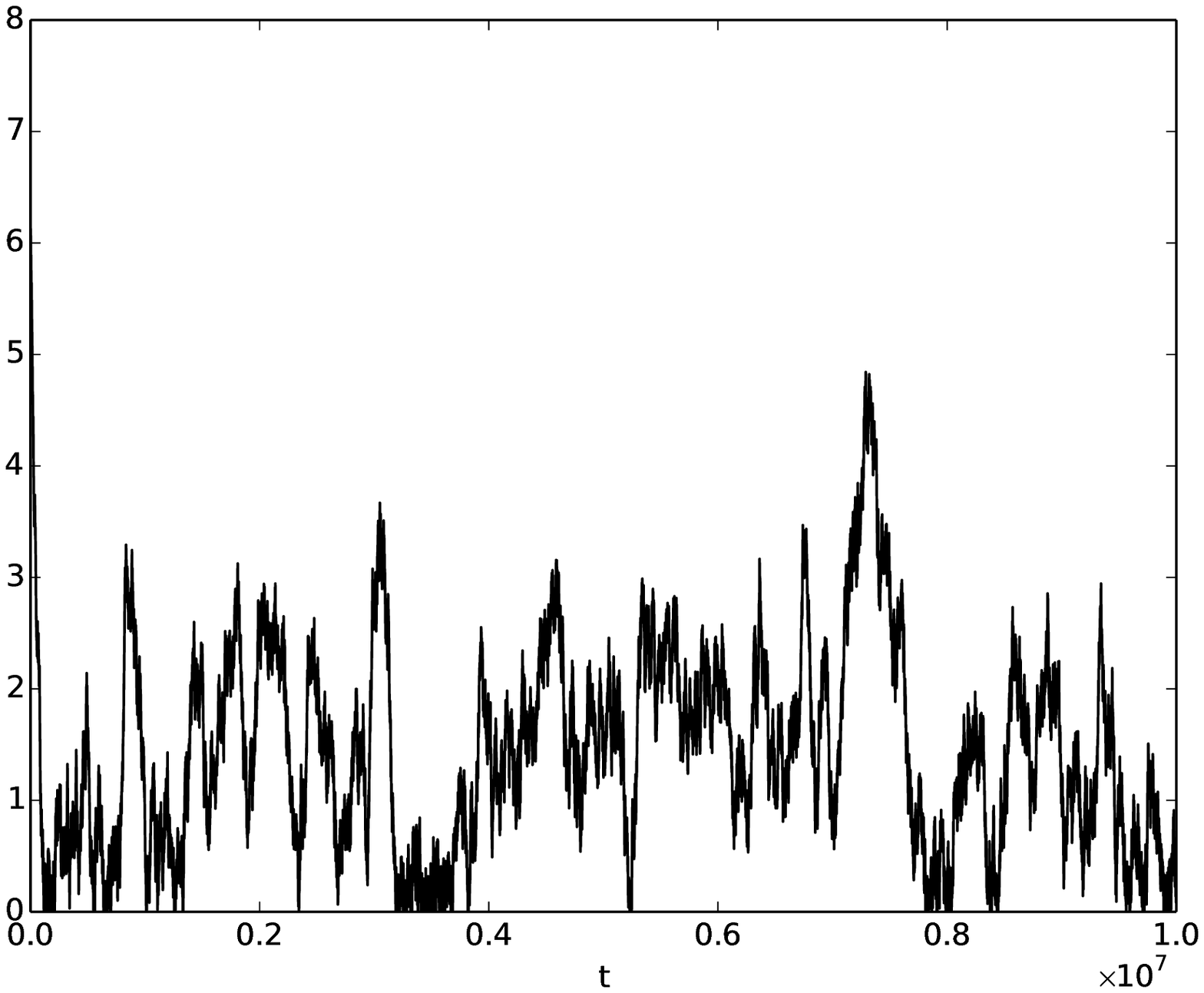}
\caption{Idle servers}
\end{subfigure}
\begin{subfigure}{0.5\textwidth}
\includegraphics[scale = 0.4]{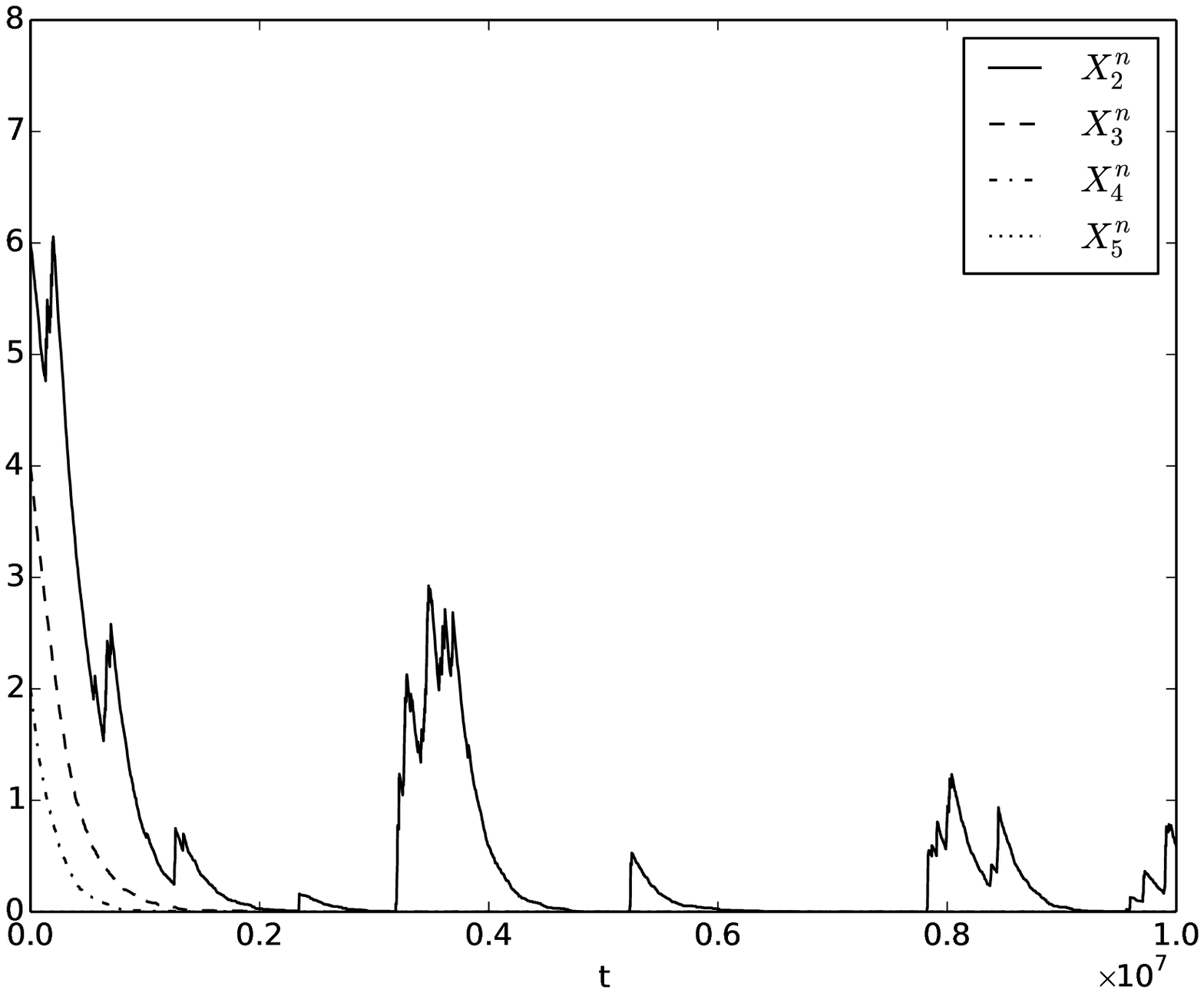}
\caption{Longer queues}
\end{subfigure}
\caption{A simulated sample path of an $M/M/n$-JSQ system, showing the scaled number of idle servers (a) and queues of length at least two ($X_2^n$), three ($X_3^n$), four ($X_4^n$), and five ($X_5^n$). Simulated with $n = 10^5$, $\beta = 2.0$.
}\label{fig:full-start}
\end{figure}

Now we make this more precise. Consider a system in which any arrival that would create a queue of length 3 or longer is rejected. That is, if an arrival occurs when all queues contain at least two customers, that arriving customer does not enter the system. We will denote this system $\hat{Q}^n = (\hat{Q}_1^n,\hat{Q}_2^n,\ldots)$. We also introduce scaled versions
\begin{equation}\label{eq:X-hat-def}
\hat{X}_1^n(t) = \frac{\hat{Q}_1^n(t) - n}{\sqrt{n}} \quad \text{ and } \quad \hat{X}_i^n(t) = \frac{\hat{Q}_i^n(t)}{\sqrt{n}}, \quad i \ge 2.
\end{equation}

It will also be convenient for us to adopt the condition from Theorem \ref{thm:trunc-result} that $\hat{Q}_{k+1}^n(0) = 0$ for some $k \ge 2$. Since queues of length 3 or longer are never created in this system, this implies $\hat{Q}_{i}^n(t) = 0$ for $i \ge k+1$, and thus analysis of $\hat{Q}^n(t)$ can be limited to the first $k$ dimensions. 

We will now construct the truncated process $\hat{X}^n(t)$ and show that it has the integral form in Theorem \ref{thm:app-int}. Our representation will be similar to the first martingale representation of \cite{martingale}; in particular it will rely upon random time changes of rate-1 Poisson processes.

\subsection{Random time change.}
We let $A, D_i$ for $1 \le i \le k$ be rate-1 Poisson processes and write
\begin{align}
\hat{Q}^n_1(t) &= Q^n_1(0) + A\left(\lambda_n n t\right) - D_1\left(\int_0^t\left(\hat{Q}^n_1(s) - \hat{Q}^n_2(s)\right)ds\right) - \hat{U}^n_1(t), \label{eq:app-pois-1}\\
\hat{Q}^n_2(t) &= Q^n_2(0) + \hat{U}^n_1(t) - D_2\left(\int_0^t\left(\hat{Q}^n_2(s) - \hat{Q}^n_3(s)\right) ds\right) - \hat{U}^n_2(t),\label{eq:app-pois-2} \\
\hat{Q}^n_i(t) &= Q^n_i(0) - D_i\left(\int_0^t\left(\hat{Q}^n_i(s) - \hat{Q}^n_{i+1}(s)\right) ds\right), \label{eq:app-pois-3} \\
\hat{Q}^n_k(t) &= Q^n_k(0) - D_k\left(\int_0^t\hat{Q}^n_k(s)ds\right), \label{eq:app-pois-4} 
\end{align}
where $\hat{U}_1^n(t)$ is the number of arrivals in $[0,t]$ when every server has at least one customer, and $\hat{U}_2^n(t)$ is the number of arrivals in $[0,t]$ when every server has at least one customer \emph{and} all $n$ servers have two customers. Formally, we define
\begin{align}
\hat{U}_1^n(t) &= \int_0^t\1\left\{\hat{Q}_1^n(s) = n\right\}dA(\lambda_nns), \label{eq:trunc-u1}\\
\hat{U}_2^n(t) &= \int_0^t\1\left\{\hat{Q}_1^n(s) = n, \hat{Q}_2^n(s) = n\right\}dA(\lambda_nns). \label{eq:trunc-u2}
\end{align}
We can understand \eqref{eq:app-pois-1} term-by-term: first we record the initial state of the system with $Q_1^n(0)$, then arrivals are counted at their full rate $\lambda_n n$. The $D_1$ term represents departures, which occur as a Poisson process with rate equal to the number of customers in service. Since $\hat{Q}_1$ includes queues of length 1 and length 2, however, $\hat{Q}_1$ will only decrease when a customer departs a queue and leaves the server empty. Therefore the instantaneous rate at time $s$ in the $D_1$ term is $\hat{Q}_1^n(s) - \hat{Q}_2^n(s)$, the number of queues of length exactly 1 at time $s$. Through the first three terms of \eqref{eq:app-pois-1} we have recorded what the value of $\hat{Q}_1$ would be if it were not constrained to be at most $n$, so the final term will represent this barrier. The process $\hat{U}_1^n$ records any arrival which would increase $\hat{Q}_1$ above $n$, balancing the overcounting we get from $A(\lambda_nnt)$.

We can understand \eqref{eq:app-pois-2} in much the same way, with the key difference being in the arrival process. Since arriving customers will always join the shortest available queue, the number of length 2 queues will increase only when all servers are busy. Such arrivals are exactly recorded by $\hat{U}_1^n$, so this will be the process we use to record potential increases to $\hat{Q}_2^n$. The process $\hat{U}_2^n$ provides the upper barrier $n$ on $\hat{Q}_2^n$.

The remaining equations \eqref{eq:app-pois-3}-\eqref{eq:app-pois-4} are the same except that we record no arrivals, as our truncated approximation does not create queues of length 3 or longer.

As in \cite[Lemma 2.1]{martingale}, we can verify that this construction is well defined and generates an element of $D^k$ by conditioning on the starting state $Q^n(0)$ and processes $A,D_i$ then constructing recursively.

\subsection{Martingales.}\label{sec:mart}
Because our approach to \eqref{eq:app-pois-1}-\eqref{eq:app-pois-4} will be to apply the functional central limit theorem (FCLT) for Poisson processes, we will now rewrite the time changes of Poisson processes as time changes of scaled Poisson processes. To that end, we define scaled martingales
\begin{align}
\hat{M}_0^{n}(t) &= \frac{1}{\sqrt{n}}A\left(\lambda_n n t\right) - \lambda_n \sqrt{n}t, \label{eq:mart-0}\\
\hat{M}_i^{n}(t) &= \frac{1}{\sqrt{n}}D_i\left(\int_0^t\left(\hat{Q}^n_{i}(s) - \hat{Q}^n_{i+1}(s)\right)ds\right) - \frac{1}{\sqrt{n}}\int_0^t\left(\hat{Q}^n_{i}(s) - \hat{Q}^n_{i+1}(s)\right)ds, \quad 1 \le i < k, \label{eq:mart-i}\\
\hat{M}_k^{n}(t) &= \frac{1}{\sqrt{n}}D_k\left(\int_0^t\hat{Q}^n_{k}(s)ds\right) - \frac{1}{\sqrt{n}}\int_0^t \hat{Q}^n_{k}(s) ds. \label{eq:mart-k}
\end{align}
Via an argument exactly analogous to that of in \S 7.1 of \cite{martingale} leading to Theorem 7.2 we obtain that $\hat{M}_i^{n}$ for $0 \le i \le k$ are square-integrable martingales with respect to an appropriate filtration. We note for later use that this argument also supplies the predictable quadratic variations
\begin{align}
\left\langle \hat{M}_0^{n}\right\rangle(t) &= \lambda_n t, \label{eq:trunc-qv-1}\\
\left\langle \hat{M}_i^{n}\right\rangle(t) &= \frac{1}{n}\int_0^t(\hat{Q}^n_i(s)-\hat{Q}^n_{i+1}(s))ds, \label{eq:trunc-qv-2}\\
\left\langle \hat{M}_k^{n}\right\rangle(t) &= \frac{1}{n}\int_0^t\hat{Q}^n_k(s)ds. \label{eq:trunc-qv-3}
\end{align}

We also define
\[
\hat{V}_1^n(t) = \frac{\hat{U}_1^n(t)}{\sqrt{n}} \quad\text{ and }\quad \hat{V}_2^n(t)  = \frac{\hat{U}_2^n(t)}{\sqrt{n}}.
\]
Then we have
\begin{align}
\hat{X}^n_1(t) &= \frac{\hat{Q}^n_1(t) - n}{\sqrt{n}}  \notag\\
&= \frac{Q^n_1(0) - n}{\sqrt{n}} + \frac{1}{\sqrt{n}}A\left(\lambda_n n t \right) \notag\\
&\quad\quad- \frac{1}{\sqrt{n}}D_1\left(\int_0^t(\hat{Q}^n_1(s)-\hat{Q}^n_2(s))ds\right) - \frac{\hat{U}_1^n(t)}{\sqrt{n}} \notag\\
&=X^n_1(0) + \hat{M}_0^{n}(t) + \lambda_n\sqrt{n}t - \hat{V}_1^n(t) \notag\\
&\quad\quad- \hat{M}_1^{n}(t)- \frac{1}{\sqrt{n}}\int_0^t\left(\hat{Q}^n_1(s) - \hat{Q}^n_2(s)\right)ds \notag\\
&= X^n_1(0) + \hat{M}_0^{n}(t) - \hat{M}_1^{n}(t) + \lambda_n\sqrt{n}t - \hat{V}_1^n(t) \notag\\
&\quad\quad -\sqrt{n}t- \int_0^t\left(\frac{\hat{Q}^n_1(s)-n}{\sqrt{n}}-\frac{\hat{Q}^n_2(s)}{\sqrt{n}}\right)ds \notag\\
&=X^n_1(0) + \hat{M}_0^{n}(t) - \hat{M}_1^{n}(t) - (1 - \lambda_n)\sqrt{n}t \label{eq:app-mart-1}\\
&\quad\quad- \int_0^t(\hat{X}^n_1(s)-\hat{X}^n_2(s))ds - \hat{V}_1^n(t), \notag
\end{align}
and
\begin{align}
\hat{X}^n_2(t) &= X^n_2(0) + \hat{V}_1^n(t) - \hat{M}_2^{n}(t) - \int_0^t\left(\hat{X}^n_2(s)-\hat{X}_3^n(s)\right)ds - \hat{V}_2^n(t), \label{eq:app-mart-2} \\
\hat{X}^n_i(t) &= X^n_i(0) - \hat{M}_i^{n}(t) - \int_0^t\left(\hat{X}^n_i(s)-\hat{X}_{i+1}^n(s)\right)ds, \quad 3 \le i \le k-1, \label{eq:app-mart-3}\\
\hat{X}^n_k(t) &= X^n_k(0) - \hat{M}_k^{n}(t) - \int_0^t\hat{X}^n_k(s)ds. \label{eq:app-mart-4}
\end{align}

At this point we can also note that \eqref{eq:app-mart-1}-\eqref{eq:app-mart-4} put $\hat{X}^n(t)$ in the integral form of Theorem \ref{thm:app-int}. The only difference is the processes $\hat{V}^n$, which are not described in exactly the same way. We see, however, that by \eqref{eq:trunc-u1} we have
\begin{align}
0 &= \int_0^\infty \1\{\hat{Q}_1^n(s) < n\}d\hat{U}_1^n(s) \notag\\
&= \int_0^\infty \1\{\hat{X}_1^n(s) < 0\}d\hat{U}_1^n(s) \notag\\
&= \int_0^\infty \1\{\hat{X}_1^n(s) < 0\}d\hat{V}_1^n(s).\label{eq:v1-valid}
\end{align}
Similarly by \eqref{eq:trunc-u2} we have
\begin{align}
0 &= \int_0^\infty \1\{\hat{Q}_1^n(s) < n \text{ or } \hat{Q}_2^n(s) < n\}d\hat{U}_2^n(t). \notag
\end{align}
which implies
\begin{align}
0 &= \int_0^\infty \1\{\hat{Q}_2^n(s) < n\}d\hat{U}_2^n(t) \notag\\
&= \int_0^\infty \1\{\hat{X}_2^n(s) < \sqrt{n}\}d\hat{V}_2^n(t). \label{eq:v2-valid}
\end{align}

Thus by \eqref{eq:app-mart-1}-\eqref{eq:app-mart-4} and \eqref{eq:v1-valid}-\eqref{eq:v2-valid} $\hat{X}^n$ is the unique solution of \eqref{eq:int-app-1}-\eqref{eq:int-app-3} for $b = X^n(0)$, $y_1 = \hat{M}_0^n(t)-\hat{M}_1^n(t)-(1-\lambda_n)\sqrt{n}t$, $y_i = -\hat{M}_i^n(t)$, $2 \le i \le k$, and $B = \sqrt{n}$. Thus to apply the CMT it remains to prove the convergence of the martingales \eqref{eq:mart-0}-\eqref{eq:mart-k}.

\section{Martingale convergence.}\label{sec:mart-conv}
We will now prove the convergence of $\hat{M}_i^{n}$ to Brownian motions.
\begin{lemma}\label{thm:mart-conv}
For the sequences of scaled martingales $\hat{M}_i^{n}$ defined in Section \ref{sec:mart} we have the convergence
\begin{equation}\label{eq:mart-limit}
\left(\hat{M}_0^{n},\hat{M}_1^{n},\hat{M}_2^{n},\ldots,\hat{M}_k^{n}\right) \Rightarrow (W_1,W_2,0,\ldots,0) \quad \text{ in } D\quad \text{ as } n \to \infty,
\end{equation}
where $W_1$ and $W_2$ are independent standard Brownian motions.
\end{lemma}

To prove this lemma we will rely upon the CMT and the FCLT for Poisson processes (\cite[Theorem 4.2]{martingale}), which we restate here convenience.
\begin{lemma}{(FCLT for independent Poisson processes)}\label{thm:fclt}
If $A$ and $D_i$ are independent rate-1 Poisson processes and
\[M_{C,n}(t) = \frac{C(nt) - nt}{\sqrt{n}}\]
for $C = A, D_i$, then
\[(M_{A,n},M_{D_1,n},\ldots,M_{D_k,n}) \Rightarrow (W_1,W_2,\ldots,W_{k+1}) \quad \text{ in } D^{k+1} \quad \text{ as } n \to \infty\]
where $W_i$ are independent standard Brownian motions.
\end{lemma}

To apply Lemma \ref{thm:fclt} we will define random and deterministic time changes such that the martingales $\hat{M}_i^{n}$ can be written as a composition of a time change and the scaled Poisson processes $M_{C,n}$. Specifically, let
\begin{align}
\Phi_{A,n}(t) &= \lambda_n t, \label{eq:trunc-change-1}\\
\Phi_{D_i,n}(t) &= \frac{1}{n}\int_0^t\hat{Q}^n_i(s)ds-\frac{1}{n}\int_0^t\hat{Q}^n_{i+1}(s)ds, \label{eq:trunc-change-2}\\
\Phi_{D_k,n}(t) &= \frac{1}{n}\int_0^t\hat{Q}^n_{k}(s)ds, \label{eq:trunc-change-3}
\end{align}
so that we have
\[\hat{M}_0^{n} = \hat{M}_{A,n} \circ \Phi_{A,n}, \quad \hat{M}_i^{n} = \hat{M}_{D_i,n} \circ \Phi_{D_i,n}, \quad 1 \le i \le k.\]
To apply the CMT with the composition map $\circ$, we need to determine the limits of the time changes \eqref{eq:trunc-change-1}-\eqref{eq:trunc-change-3}.

First we note that \eqref{eq:lambda-scale} implies $\lambda_n \to 1$, which in turn implies
\begin{equation}\label{eq:trunc-change-limit-1}
\Phi_{A,n} \Rightarrow e \quad \text{ as } n \to \infty,
\end{equation}
where $e$ is the identity function in $D$.

Next we note that the terms of \eqref{eq:trunc-change-2} interleave over $i$, so for $1 \le i \le k-1$ we have
\[\Phi_{D_i,n} \Rightarrow f_i - f_{i+1} \text{ as } n \to \infty,\]
where $f_i$ is the limit of $\tilde{\Phi}_{D_i,n}$ with
\[\tilde{\Phi}_{D_i,n}(t) = \frac{1}{n}\int_0^t\hat{Q}^n_i(s)ds.\]
To find $f_i$ we will first show fluid limits for $\hat{Q}_i^n$.
\begin{lemma}\label{thm:fluid}
Let $\Psi_i^n$ for $1 \le i \le k$ be defined by
\[\Psi^n_i(t) = \frac{\hat{Q}_i^n(t)}{n}, \quad t \ge 0.\]
Then for $2 \le i \le k$,
\begin{equation}\label{eq:trunc-psi}
\Psi^n_1 \Rightarrow \omega \quad \text{ and } \quad \Psi_i^n \Rightarrow 0 \quad \text{ as } n \to \infty
\end{equation}
where $\omega(t) = 1$ for $t \ge 0$.
\end{lemma}

The proof of this lemma is found in Section \ref{sec:fluid}. To use Lemma \ref{thm:fluid} we define a continuous function $h:D \to D$ by
\[h(x)(t) = \int_0^tx(s)ds\]
for $t \ge 0$. The $\tilde{\Phi}_{D_1,n} = h \circ \Psi_n$ so by the CMT and Lemma \ref{thm:fluid} we know $f_1 = h \circ \omega$. Namely,
\[f_1(t) = \int_0^t 1 ds = t\]
so $f_1 = e$ is the identity function in $D$. Therefore we have
\[
\tilde{\Phi}_{D_1,n} \Rightarrow e \quad \text{ as } n \to \infty.
\]
For $2 \le i \le k$ we have $f_i(t) = \int_0^t0ds = 0$ so $f_i = 0$ on $D$.
We conclude that
\begin{equation}\label{eq:trunc-change-limit-2}
\Phi_{D_1,n} \Rightarrow e \quad \text{ as } n \to \infty,
\end{equation}
and for $2 \le i \le k$
\begin{equation}\label{eq:trunc-change-limit-3}
\Phi_{D_i,n} \Rightarrow 0 \quad \text{ as } n \to \infty.
\end{equation}

Therefore once we establish Lemma \ref{thm:fluid} we can prove Lemma \ref{thm:mart-conv}:
\begin{proof}[Proof of Lemma \ref{thm:mart-conv}.]
We apply the CMT with Lemma \ref{thm:fclt} and the limits \eqref{eq:trunc-change-limit-1}, \eqref{eq:trunc-change-limit-2}, and \eqref{eq:trunc-change-limit-3} to obtain
\begin{align}
\left(\hat{M}_0^{n},\hat{M}_1^{n},\hat{M}_2^{n}, \ldots, \hat{M}_k^{n}\right) &= \left(M_{A,n} \circ \Phi_{A,n},M_{D_1,n} \circ \Phi_{D_1,n}, M_{D_2,n} \circ \Phi_{D_2,n}, \ldots, M_{D_k,n} \circ \Phi_{D_k,n}\right) \notag\\
&\Rightarrow \left(W_1 \circ e, W_2 \circ e, W_3 \circ 0, \ldots, W_{k+1} \circ 0\right) \notag\\
&= (W_1,W_2,0, \ldots, 0). \notag
\end{align}
\end{proof}

\subsection{Fluid limit.}\label{sec:fluid}
We will prove Lemma \ref{thm:fluid} by showing that $\hat{X}^n$ is stochastically bounded in $D$. Namely, we will prove that the sequence of real-valued random variables $\vn{\hat{X}_i^n}_t$ is tight for every $t > 0$ and $i \ge 1$. For a more complete discussion of stochastic boundedness as we use it here see \S 5 of \cite{martingale}.

The stochastic boundedness of $\hat{X}_1^n$ and $\hat{X}_2^n$ will follow from the stochastic boundedness of $\hat{M}_0^{n}$, $\hat{M}_1^{n}$, $\hat{M}_2^{n}$, and $\hat{X}_3^n$. To see this, we prove
\begin{lemma}\label{thm:stoch-bound}
Given $(B_n,X_i^n(0),Y_i^{n})$ a random element of $\Rbar_+ \times \R \times D$ for each $n \ge 1$ and $i = 1,2$, recall that Lemma \ref{thm:trunc-int-low} implies that the system
\begin{align*}
\hat{X}_1^n(t) &= X_1^n(0) + Y^{n}_1(t) + \int_0^t(-\hat{X}_1^n(s)+\hat{X}_2^n(s))ds - \hat{V}_1^n(t), \\
\hat{X}_2^n(t) &= X_2^n(0) + Y^{n}_2(t) + \int_0^t(-\hat{X}_2^n(s))ds +\hat{V}_1^n(t)- \hat{V}_2^n(t) \ge 0, \\
0 &= \int_0^\infty \1\{\hat{X}_1^n(t) < 0\}d\hat{V}_1^n(t), \\
0 &= \int_0^\infty \1\{\hat{X}_2^n(t) < B_n\}d\hat{V}_2^n(t),
\end{align*}
has a unique solution $(\hat{X}^n,\hat{V}^n)$. If the sequences $(X^n(0) , n \ge 1)$ and $(Y^{n}_i , n \ge 1)$ are stochastically bounded for $i = 1,2$, then the sequence $(\hat{X}^n , n \ge 1)$ is stochastically bounded in $D$.
\end{lemma}

Note that we do not require boundedness for $B_n$.

\begin{proof}[Proof.] 
We fix $t > 0$. We will establish the bound
\begin{equation}\label{eq:stoch-bound}
\vn{\hat{X}^n}_t \le 8e^{6t}\left(|X^n(0)| + \vn{Y^{n}}_t\right),
\end{equation}
from which the result follows.

To show \eqref{eq:stoch-bound}, we will prove a similar bound for the unreflected process $W^n$ defined by Lemma \ref{thm:app-int-unrefl}. Then \eqref{eq:stoch-bound} will follow from the Lipschitz continuity of the reflection maps $\phi_0$ and $\phi_{B_n}$.

Just as in Lemma \ref{thm:trunc-int-low} and Lemma \ref{thm:app-int-unrefl} we write $\hat{X}_1^n(t) = \phi_0(W_1^n(t))$ and $\hat{X}_2^n(t) = \phi_{B^n}(W_2^n(t))$ where $W_1^n(t)$ and $W_2^n(t)$ satisfy
\begin{align}
W_1^n(t) &=   X_1^n(0) + Y^{n}_1(t)+ \int_0^t(-\phi_0(W_1^n(s))+\phi_{B^n}(W_2^n(s)))ds, \label{eq:stoch-w-1} \\
W_2^n(t) &= X_2^n(0) + Y^{n}_2(t) + \int_0^t(-\phi_{B^n}(W_2^n(s)))ds +\psi_0\left(W_1^n(t)\right) \label{eq:stoch-w-2}.
\end{align}

We now use Gronwall's inequality as stated in Lemma \ref{thm:gronwall}. Using the Lipschitz property for $\phi_0$,$\phi_{B_n}$, and $\psi_0$ we have for $t \ge 0$
\begin{align*}
\vn{W_1^n}_t &\le |X_1^n(0)| + \vn{Y_1^{n}}_t + 2\int_0^t\left(\vn{W^n_2}_s + \vn{W^n_1}_s\right)ds, \\
\vn{W_2^n}_t &\le |X_2^n(0)| + \vn{Y_2^{n}}_t + \vn{\psi_0(W_1^n)}_t + \int_0^t\vn{W_2^n}_sds.
\end{align*}
Now we note that we have
\[\vn{\psi_0(W_1^n)}_t \le \vn{W_1^n}_t.\]
We define
\begin{align*}
u_1(t) &= \vn{W_1^n}_t \quad\text{ and }\quad u_2(t) = \left(\vn{W_2^n}_t - \vn{W_1^n}_t\right)^+.
\end{align*}
Finally we note
\begin{equation}\label{eq:stoch-u2-bound}
\vn{W^n_2}_t \le u_2(t) + u_1(t),
\end{equation}
so we can write the inequalities
\begin{align*}
u_1(t) &\le |X_1^n(0)| + \vn{Y_1^{n}}_t + 4\int_0^tu_1(s)ds + 2\int_0^tu_2(s)ds,\\
u_2(t) &\le |X_2^n(0)| + \vn{Y_2^{n}}_t + \int_0^tu_1(s)ds + \int_0^tu_2(s)ds.
\end{align*}
Let $|X^n(0)| + \vn{Y^{n}}_t = K$. Then Lemma \ref{thm:gronwall} implies
\[u_1(t) \le 2Ke^{6t} \quad \text{ and }\quad u_2(t) \le 2Ke^{6t}.\]
From \eqref{eq:stoch-u2-bound} and the definitions of $u_1$ and $u_2$ we obtain
\[\vn{W_1^n}_t \le 2Ke^{6t} \quad \text{ and } \quad \vn{W_2^n}_t \le 4Ke^{6t}.\]
Since $\phi_0$ and $\phi_{B^n}$ are Lipschitz continuous with constant $2$ this implies
\[\vn{\hat{X}_1^n}_t \le 4Ke^{6t} \quad \text{ and } \quad \vn{\hat{X}_2^n}_t \le 8Ke^{6t},\]
which proves \eqref{eq:stoch-bound}.
\end{proof}
Note that this proof also provides the boundedness of $w$ that we use in the proof of continuity of $w$ in Lemma \ref{thm:app-int-unrefl}, and that it does not use any of the continuity properties proved using that boundedness.

In our application of Lemma \ref{thm:stoch-bound} we will have $Y_1^{n} = \hat{M}_1^{n,1}-\hat{M}_1^{n,2} - (1-\lambda_n)\sqrt{n}t$ and $Y_2^n(t) = \hat{M}_2^{n,2}(t) + \int_0^t\hat{X}_3^n(s)ds$, so it remains to prove that $ \int_0^t\hat{X}_3^n(s)ds$ and each martingale $\hat{M}_k^{n,i}$ is stochastically bounded. 

To show that $\int_0^t\hat{X}_3^n(s)ds$ is stochastically bounded we need only show that $\hat{X}_3^n$ is stochastically bounded. This follows from the fact that $Q_3^n(0)$ is stochastically bounded by assumption \eqref{eq:trunc-initial} and 
\[\hat{X}^n_3(t) = \frac{\hat{Q}_3^n(t)}{\sqrt{n}} \le \frac{Q_3^n(0)}{\sqrt{n}}\]
because no queues of length 3 or longer are ever created. An identical argument proves stochastic boundedness of $\hat{X}^n_i$ for $4 \le i \le k$.

To prove the stochastic boundedness of these martingales we will use the following lemma from \cite{martingale}:
\begin{lemma}[\cite{martingale} Lemma 5.8]
Suppose that, for each $n \ge 1$, $M_n$ is a square integrable martingale with predictable quadratic variation $\langle M_n \rangle$. If the sequence of random variables $\langle M_n\rangle(T)$ is stochastically bounded in $\R$ for each $T > 0$, then the sequence of stochastic processes $M_n$ is stochastically bounded in $D$.
\end{lemma}
We now prove that the predictable quadratic variations of $\hat{M}_i^{n}$ are stochastically bounded. In the case of $\hat{M}_0^{n}$ this is immediate since by \eqref{eq:trunc-qv-1} the quadratic variation is deterministic.

For $\hat{M}_1^{n}$ we refer to \eqref{eq:trunc-qv-2} and apply crude bounds to see
\begin{align*}
\left\langle \hat{M}_1^{n}\right\rangle(t) &= \frac{1}{n}\int_0^t(\hat{Q}_1^n(s)-\hat{Q}_2^n(s))ds \\
&\le \frac{1}{n}\int_0^t\hat{Q}_1^n(s)ds \\
&\le \frac{t}{n}\left(Q_1^n(0) + A(\lambda_nnt)\right).
\end{align*}
It suffices to show stochastic boundedness of each term in the sum. For $Q_1^n(0)$ this follows from assumption \eqref{eq:trunc-initial}.

For $A(\lambda_nnt)$ we note $\lambda_n \to 1$ so by the strong law of large numbers (SLLN) for Poisson processes we have
\[\frac{A(\lambda_nnt)}{n} \to e(t)\]
with probability 1, which implies stochastic boundedness, so we conclude that $\hat{M}_1^{n}$ is stochastically bounded.

For $\hat{M}_2^{n}$ we have
\begin{align*}
\left\langle \hat{M}_2^{n}\right\rangle(t) &\le \frac{t}{n}\left(Q_2^n(0) + \hat{U}_1^n(t)\right) \\
&\le \frac{t}{n}\left(Q_2^n(0) + A(\lambda_n n t)\right),
\end{align*}
and stochastic boundedness follows.

We now return to the proof of Lemma \ref{thm:fluid}:

\begin{proof}[Proof of Lemma \ref{thm:fluid}.]
We have for $i \le 2 \le k$ that
\[\hat{X}_1^n = \frac{\hat{Q}_1^n - n}{\sqrt{n}} \quad \text{ and } \quad \hat{X}_i^n = \frac{\hat{Q}_i^n}{\sqrt{n}}\]
are stochastically bounded. Therefore
\[\frac{\hat{X}_i^n}{\sqrt{n}} \Rightarrow 0 \quad \text{ in } D\quad \text{ as } n \to \infty.\]
From the definition of $\hat{X}^n$ this is equivalent to
\[\Psi_1^n = \frac{\hat{Q}_1^n}{n} \Rightarrow \omega \quad \text{ and }\quad \Psi_i^n = \frac{\hat{Q}_i^n}{n} \Rightarrow 0 \quad \text{ in } D\quad \text{ as } n \to \infty\]
for $2 \le i \le k$.
\end{proof}

\subsection{Proof of Theorem \ref{thm:trunc-result}.}
Now that we have the convergence of the martingale processes $\hat{M}_i^{n}$, we can apply the CMT to prove Theorem \ref{thm:trunc-result}.

\begin{proof}[Proof of Theorem \ref{thm:trunc-result}.]
We first show $\hat{X}^n \Rightarrow X$.

In Theorem \ref{thm:app-int}, in the pre-limit regime we set $B_n = \sqrt{n}$, $b_i = X_i^n(0)$ for $1 \le i \le k$,
\[y_1(t) = \hat{M}_0^{n}(t)-\hat{M}_1^{n}(t) - (1-\lambda_n)\sqrt{n}t,\]
and $y_i(t) = -\hat{M}_i^{n}(t)$ for $2 \le i \le k$. Equations \eqref{eq:v1-valid} and \eqref{eq:v2-valid} show that $\hat{V}_1^n$ and $\hat{V}_2^n$ are appropriately acting as $u_1$ and $u_2$ in the integral representation, so $x_i(t) = \hat{X}_i^n(t)$ for $1 \le i \le k$. For application of the CMT we need only determine the limits of $B$, $b$ and $y$. We have $B_n \to \infty$, so in the limit $u_2 = 0$.

By assumption we have
\[\hat{X}_k^n(0) \Rightarrow X_k(0),\]
so in the limiting system we let $b_i = \hat{X}_i(0)$. Next we have by \eqref{eq:lambda-scale} and \eqref{eq:mart-limit}
\begin{align*}
\hat{M}_0^{n}(t)-\hat{M}_1^{n}(t) - (1-\lambda_n)\sqrt{n}t &\Rightarrow W_1(t) - W_2(t) - \beta t \\
&\de \sqrt{2}W(t) - \beta t,
\end{align*}
where $W$ is a standard Brownian motion and $\de$ indicates equivalence in distribution. Another application of \eqref{eq:mart-limit} implies
\[-\hat{M}_i^{n} \Rightarrow 0\]
for $2 \le i \le k$ so in the limiting system we have $y_1(t) = \sqrt{2}W(t) -\beta t$ and $y_i(t) = 0$, for $2 \le i \le k$.

The CMT then implies that for $1 \le i \le k$, $\hat{X}_i^n \Rightarrow X_i$ in $D$ as $n \to \infty$ where $X_i$ is described by \eqref{eq:trunc-limit-1}-\eqref{eq:trunc-reflection}. We can augment the truncated system with $\hat{X}^n_i(t) = 0$ for $i > k$ and note that $0 = \hat{X}^n_i \Rightarrow X_i = 0$ for such $i$. Therefore $\hat{X}^n \Rightarrow X$.

Now we consider the untruncated system described by $X^n$. By an argument like that Section \ref{sec:mart-rep}, we have
\begin{align}
Q^n_1(t) &= Q^n_1(0) + A\left(\lambda_n n t\right) - D_1\left(\int_0^t\left(Q^n_1(s) - Q^n_2(s)\right)ds\right) - U^n_1(t),\label{eq:full-q-1}\\
Q^n_i(t) &= Q^n_i(0) + U^n_{i-1}(t) - D_i\left(\int_0^t\left(Q^n_i(s) -Q^n_{i+1}(s)\right) ds\right) - U^n_i(t), \quad i \ge 2, \label{eq:full-q-i}
\end{align}
where $U^n_i(t)$ is the number of arrivals in $[0,t]$ when every server has at least $i$ customers. Note that we introduce extra rate one Poisson processes $D_i$ for $i > k$.

Now define
\begin{equation}\label{eq:stop-def}
t_n^* = \inf\{t \ge 0 : Q^n_2(t) = n\}
\end{equation}
and note that for $t \in [0,t_n^*)$ we have $U_i^n(t) = 0$ for $i \ge 2$. This, along with $Q^n_i(0) = 0$ for $i >k$, implies that for such $t$, the system \eqref{eq:full-q-1}-\eqref{eq:full-q-i} becomes
\begin{align*}
Q^n_1(t) &= Q^n_1(0) + A\left(\lambda_n n t\right) - D_1\left(\int_0^t\left(Q^n_1(s) - Q^n_2(s)\right)ds\right) - U^n_1(t),\\
Q^n_2(t) &= Q^n_2(0) + U^n_1(t) - D_2\left(\int_0^t\left(Q^n_2(s) - Q^n_3(s)\right) ds\right) - U^n_2(t), \\
Q^n_i(t) &= Q^n_i(0) - D_i\left(\int_0^t\left(Q^n_i(s) - Q^n_{i+1}(s)\right) ds\right), \\
Q^n_k(t) &= Q^n_k(0) - D_k\left(\int_0^tQ^n_k(s)ds\right), 
\end{align*}
which precisely matches \eqref{eq:app-pois-1}-\eqref{eq:app-pois-4}. Thus for $t \in [0,t_n^*)$, $X^n(t)$ and $\hat{X}^n(t)$ are identical.

It only remains to show for all $t \ge 0$ that $\PP(t_n^* \le t) \to 0$ as $n \to \infty$. Because the systems are identical up to time $t_n^*$, we can replace $Q^n_2(t)$ in \eqref{eq:stop-def} with $\hat{Q}^n_2(t)$ to see
\begin{align*}
\PP(t_n^* \le t) &=  \PP\left(\sup_{0 \le s \le t}\hat{Q}^n_2(s) \ge n\right) = \PP\left(\sup_{0 \le s \le t}\hat{X}^n_2(s) \ge \sqrt{n}\right) \\
&\le \PP\left(\sup_{0 \le s \le t}\hat{X}^n_2(s) \ge C\right)
\end{align*}
for constant $0 < C \le \sqrt{n}$. By the weak convergence $\hat{X}_2^n \Rightarrow X_2$ and the fact that $\left\{\sup_{0 \le s \le t}\hat{X}^n_2(s) \ge C\right\}$ is closed, we have
\begin{align*}
\limsup_n \PP\left(\sup_{0 \le s \le t}\hat{X}^n_2(s) \ge C\right) &\le \PP\left(\sup_{0 \le s \le t}X_2(s) \ge C\right)
\end{align*}
By continuity of probability we have
\[\lim_{C \to \infty}\PP\left(\sup_{0 \le s \le t}X_2(s) \ge C\right) = \PP\left(\sup_{0 \le s \le t}X_2(s) = \infty \right) = 0.\]
We therefore have
\begin{align*}
\limsup_n\PP(t_n^* \le t) &\le \lim_{C \to \infty}\limsup_n\PP\left(\sup_{0 \le s \le t}\hat{X}^n_2(s) \ge C\right) \le \lim_{C \to \infty}\PP\left(\sup_{0 \le s \le t}X_2(s) \ge C\right) = 0
\end{align*}
and thus $\PP(t_n^* \le t) \to 0$ as $n \to \infty$. We conclude $X^n \Rightarrow X$.
\end{proof}

\section{Waiting time.}\label{sec:wait}
An important performance measure of a queueing system is the expected time that customers will have to wait before entering service. In the $M/M/n$ system with a single queue in the Halfin-Whitt regime, the expected waiting time is of the order $\bo{1/\sqrt{n}}$. We will now show that the $M/M/n$-JSQ system has the same order of aggregate waiting time in the transient regime, and thus seems to have a minimal loss of efficiency as measured by waiting time.

Notice that our representation of the system allows us to directly consider the total time any customers in the system will wait. In particular, the instantaneous number of customers waiting to be served at a given time $t$ is precisely $\sum_{i \ge 2} Q_i^n(t)$. This quantity can be integrated over time to compute the aggregate waiting time in the system. With this insight, we prove the following:

\begin{theorem}\label{thm:ag-wait}
The aggregate waiting time of customers who arrived over the time period $[0,t]$, which we denote $Z_t^n$, satisfies
\begin{equation}\label{eq:agg-wait}
	\lim_{C \to \infty}\limsup_n\PP(Z_t^n \ge C\sqrt{n}) = 0.
\end{equation} 
Since the total number of arrivals to the system in this time is $\bth{n}$ with high probability, the waiting time per arrival is $\bo{1/\sqrt{n}}$ with high probability.
\end{theorem}

\begin{proof}
As noted above, the aggregate waiting time over the period $[0,t]$ is
\[Z^n_t = \int_0^t\sum_{i \ge 2}Q_i^n(s)ds.\]
We consider a scaled version $Y^n_t = Z^n_t / \sqrt{n}$. Note that $Y^n_t \le t  \sup_{0 \le s \le t}\sum_{i \ge 2}X_i^n(s)$, and thus
\begin{align*}
\PP\left(Y^n_t \ge C\right) &\le \PP\left( t  \sup_{0 \le s \le t}\sum_{i \ge 2}X_i^n(s) \ge C \right) 
\end{align*}
for any constant $C > 0$. By the weak convergence $X_i^n \Rightarrow X_i$ and the fact that $\{t  \sup_{0 \le s \le t}\sum_{i \ge 2}X_i^n(s) \ge C\}$ is closed, we have
\[\limsup_n \PP\left( t  \sup_{0 \le s \le t}\sum_{i \ge 2}X_i^n(s) \ge C \right) \le \PP\left( t  \sup_{0 \le s \le t}\sum_{i \ge 2}X_i(s) \ge C \right)\]
By continuity of probability and $X_i = 0$ for $i > k$ we have
\begin{align*}
\lim_{C \to \infty} \PP\left( t \sup_{0 \le s \le t}\sum_{i \ge 2}X_i(s) \ge C \right) = 0.
\end{align*}
Therefore we conclude
\[\lim_{C \to \infty}\limsup_n \PP\left(Y^n_t \ge C\right) \le \lim_{C \to \infty} \PP\left( t \sup_{0 \le s \le t}\sum_{i \ge 2}X_i(s) \ge C \right) = 0,\]
which proves \eqref{eq:agg-wait}. We note that this implies $Y^n_t = \bo{1}$ with high probability, and thus $Z^n_t = \bo{\sqrt{n}}$ with high probability.

Because customers arrive according to a Poisson process with rate $\lambda_n n = \bth{n}$, this implies the waiting time per arrival is $\bo{\sqrt{n}/n} = \bo{1/\sqrt{n}}$ with high probability, completing the proof.
\end{proof}

Theorem \ref{thm:ag-wait} does not directly tell us anything about the distribution of the waiting time. We can see from considering the system that this waiting time is distributed in a qualitatively different way than the standard $M/M/n$ system. Customers immediately enter service if there are any idle servers and otherwise wait a constant order amount of time for the previous customer in their queue to finish service. Because the aggregate waiting time is of the order $\bo{\sqrt{n}}$ and and any arriving customers who wait at all incur a constant order waiting time, the total number of customers who have to wait is also $\bo{\sqrt{n}}$. As noted above, the total number of arriving customers is order $n$, so the fraction of customers who have to wait is of the order $\bo{1/\sqrt{n}}$.

\section{Open questions.}\label{sec:discuss}
Theorem \ref{thm:trunc-result} proves that the behavior of the $M/M/n$-JSQ system in the Halfin-Whitt regime is best understood on the order of $\bo{\sqrt{n}}$. In particular, the numbers of idle servers and waiting customers will both be $\bo{\sqrt{n}}$. If long queues are initially present, they will empty in fixed time, and no additional long queues will be created.

Significant questions remain about the steady state behavior of our system. In particular, we do not characterize the distribution of the steady state of the limiting system or show that the steady state of the $n$-th system converges to the steady state of the limiting diffusion process (interchange of limits).

%Questions of the steady state are closely related to the waiting time for the system. Because customers immediately enter service if there are any idle servers and otherwise wait a constant order amount of time for the previous customer in their queue to finish service, the waiting time can be characterized by the amount of time when the system has no idle servers. We conjecture that this time is on the order $\bo{1/\sqrt{n}}$, which would lead to an expected waiting time of the same order. This is also the order of the expected waiting time in the $M/M/n$ system, so a proof of this conjecture would demonstrate that the JSQ system has a minimal loss of efficiency as measured by expected waiting time.

Finally it is always of interest to analyze our system for general interarrival and, especially, general service times distribution. We conjecture that the qualitative behavior
established in this paper in the transient domain and the conjectures above regarding the steady-state behavior and the interchange of steady-state limits remain true in this
case as well. 

\section*{Acknowledgments.}
This work was supported by NSF grant CMMI-1335155.

\bibliographystyle{plain}
\bibliography{JSQ}

\end{document}